\documentclass[11pt,reqno]{amsart}
\usepackage{amssymb,latexsym}
\usepackage{fancyhdr,amssymb}
\usepackage{hyperref}
\usepackage{listings}
\usepackage{amsmath}
\usepackage[T1]{fontenc}
\usepackage{ucs}
\usepackage[letterpaper, top = 2cm, bottom = 2cm,right = 3cm, left = 3cm]{geometry}
\theoremstyle{theorem}
\usepackage{comment}
\newtheorem{theorem}{Theorem}[section]
\newtheorem{lemma}[theorem]{Lemma}
\newtheorem{proposition}[theorem]{Proposition}

\theoremstyle{definition}
\newtheorem{definition}[theorem]{Definition}

\theoremstyle{remark}
\newtheorem{remark}{Remark}

\theoremstyle{example}
\newtheorem{example}[theorem]{Example}

\title[Fixed points of automorphisms...]{Fixed points of automorphisms of certain non-cyclic $p$-groups and the dihedral group}
\author{Akhtar Abbas, Umar Hayat, and Daniel L\'{o}pez-Aguayo}
\address{Department of Mathematics, Quaid-i-Azam University, Islamabad, Pakistan.}
\email{achtarabas@gmail.com}
\email{umar.hayat@qau.edu.pk}

\address{Tecnologico de Monterrey, Escuela de Ingenier\'{i}a y Ciencias, Hidalgo, M\'{e}xico.}
\email{dlopez.aguayo@itesm.mx}

\begin{document}
\maketitle
\begin{abstract} Let $G=\mathbf{Z}_{p} \oplus \mathbf{Z}_{p^2}$, where $p$ is a prime number. Suppose that $d$ is a divisor of the order of $G$. In this paper we find the number of automorphisms of $G$ fixing $d$ elements of $G$, and denote it by $\theta(G,d)$. As a consequence, we prove a conjecture of Checco-Darling-Longfield-Wisdom. We also find the exact number of fixed-point-free automorphisms of the group $\mathbf{Z}_{p^{a}} \oplus \mathbf{Z}_{p^{b}}$, where $a$ and $b$ are positive integers with $a<b$. Finally, we compute $\theta(D_{2q},d)$, where $D_{2q}$ is the dihedral group of order $2q$, $q$ is an odd prime and $d \in \{1,q,2q\}$.
\end{abstract}

\section{Introduction}
Automorphisms of groups, algebras, Lie algebras and codes are of fundamental importance in many areas of mathematics and other disciplines; see, for more details, \cite{3,6,7}. In \cite{5}, Farmakis and Moskowitz explain the importance of the fixed point theorems and their applications  in many areas of mathematics including, but not limited to, analysis, algebraic groups, number theory, complex analysis and group theory. For example (cf. \cite[theorem 5.4.1]{5}), if $\alpha$ is an automorphism of a finite dimensional real or complex Lie algebra $\mathfrak{g}$ having $0$ as a fixed point, then the algebra $\mathfrak{g}$ is nilpotent. Needless to say, nilpotent algebras appear naturally when one studies the structure of Lie algebras. On the other hand, the set of fixed points of an automorphism may  contain useful information. For instance, in \cite{8} Gersten proved a famous conjecture proposed by Scott that states that if $\phi$ is an automorphism of a finitely generated free group, then the set of fixed points of $\phi$ is finitely generated.
Also, in areas such as image watermarking and live video streaming, toral automorphisms are used; see, for example, \cite{7}. \\

We now describe the main results of this paper. 
Let $G$ be a finite group and let $d$ be a divisor of the order of $G$. Then we define $\theta(G,d)$ as the number of automorphisms of $G$ fixing $d$ elements of $G$. In \cite{2}, the authors discuss the fixed points of automorphisms of some finite abelian groups. They study cyclic groups and elementary abelian groups. One of their main results is the derivation of formulas for $\theta (G, d)$, where $d$ is a divisor of the order of $G$. In \cite[p.66]{2}, the following conjecture is stated.  \\

\textbf{Conjecture 1} \label{conj1}. Let $G=\mathbf{Z}_{p} \oplus \mathbf{Z}_{p^2}$, where $p$ is a prime number. Then 
\[
\theta(G,d)=
\begin{cases}
p^3(p-2)^2 & \text{for} \ d=1; \\
p(2p^3-4p^2+1) & \text{for} \ d=p; \\
p^3-p-1 & \text{for} \ d=p^2; \\
1 & \text{for} \ d=p^3.
\end{cases}
\]
In this paper, we prove the above conjecture, see Theorem \ref{theo2}. In contrast to \cite{9}, where the explicit form of the automorphism group of $\mathbf{Z}_{p} \oplus \mathbf{Z}_{p^3}$ is found, here we will make use of Theorem \ref{theo1}. This theorem, together with Lemmas \ref{lemma1}, \ref{lemma2}, \ref{lemma3}, will allow us to find the explicit matrix form of the automorphism group of $\mathbf{Z}_{p} \oplus \mathbf{Z}_{p^2}$ in a more concise way. 

In Section \ref{section4}, we find the exact number of fixed-point-free automorphisms of the group $\mathbf{Z}_{p^{a}} \oplus \mathbf{Z}_{p^{b}}$, where $a$ and $b$ are positive integers with $a<b$. Finally, in Section \ref{section5}, we compute $\theta(D_{2q},d)$; where $D_{2q}$ is the dihedral group of order $2q$, $q$ is an odd prime, and $d \in \{1,q,2q\}$. \\
Throughout this paper, we denote the cardinality of a set $X$ by $|X|$. 
\section{Preliminaries}

\begin{definition} Let $G$ be a group. The set of all automorphisms of $G$ under function composition forms a group, called the automorphism group of $G$ and it is denoted by $\operatorname{Aut}G$.
\end{definition}

\begin{definition}
For a group $G$, the fixed point group of an automorphism $f$ is defined as $$\Omega_{G}(f)=\{g\in G: f(g)=g\},$$ where $\Omega_{G}$ is a map from $\operatorname{Aut}G$ to the collection of all subgroups of $G$. The map $\Omega_{G}$ is known as the fixed point map.
\end{definition}

\begin{definition}
For each divisor $d$ of $|G|$, we denote the set of $d$-fixers as $$S_{d}^{G}=\{f\in \operatorname{Aut}G: |\Omega_{G}(f)|=d\}$$ and $\theta(G,d)=|S_{d}^{G}|$.
\end{definition}

\begin{definition} \label{defomega}
Let $G$ be a group. For a non-empty subset $X$ of $G$, we denote the set of automorphisms of $G$ fixing at least all elements of $X$, as $X_{\Omega}$. That is $$X_{\Omega}=\{f\in \operatorname{Aut}G: X\subseteq \Omega_{G}(f)\}.$$
\end{definition}

\begin{definition}
Let $G$ be a group and $H$ a subgroup of $G$. We denote the set of all automorphisms fixing $exactly$ all elements of $H$ by $\Omega^{-1}_{G}(H)$. That is
  $$ \Omega^{-1}_{G}(H)=\{f\in \operatorname{Aut}G:~\Omega_{G}(f)=H\}.$$
\end{definition}
\begin{remark}
Let $G$ be a finite group. By Lagrange's theorem, $\operatorname{Aut}G$ can be written as the disjoint union of the sets $\{\phi \in \operatorname{Aut}G: |\Omega_{G}(\phi)|=d\}$ where $d$ runs over all divisors of $|G|$. Thus, if $\{H_{i}:~1\leq i\leq n\}$ is the collection of all subgroups of order $d$, then $$\theta(G,d)=\sum_{i=1}^{n}|\Omega^{-1}_{G}(H_{i})|.$$
\end{remark}

\begin{definition} Let $n$ be a positive integer. Then the Euler's totient function $\varphi(n)$ is defined as
\begin{center}
$\varphi(n)=|\{m: m \in [1,n] \wedge (m,n)=1\}|$
\end{center}
where $(m,n)$ denotes the greatest common divisor of $m$ and $n$.
\end{definition}

\begin{definition} Let $n \geq 2$ be an integer and let $\mathbf{Z}_{n}$ denote the group of integers modulo $n$. The group of units of $\mathbf{Z}_{n}$, which is denoted by $\mathbf{Z}_{n}^{\ast}$, consists of all those congruence classes which are relatively prime to $n$; thus $|\mathbf{Z}_{n}^{\ast}|=\varphi(n)$.
\end{definition}

\begin{definition} Let $G$ be a group and let $\phi: G \rightarrow G$ be an automorphism of $G$. It is said that $\phi$ is a fixed-point-free automorphism if $\phi$ fixes only the identity element of $G$.
\end{definition}

\begin{definition} Let $G$ be a group and let $\operatorname{Aut}G$ be its automorphism group. The \emph{holomorph} of $G$ is defined as the following semi-direct product $\operatorname{Hol}(G)=G \rtimes \operatorname{Aut}G$, where the product $\ast$ is given by $(g,\alpha) \ast (h,\beta)=(g\alpha(h),\alpha \beta)$.
\end{definition}

\begin{definition} For a positive integer $n$, the dihedral group $D_{2n}$ is defined as the set of all rotations and reflections of a regular polygon with $n$ sides.
\end{definition}

\section{Proof of Conjecture 1}
In what follows, if $S$ is a subset of a group $G$, then $\langle S \rangle$ denotes the subgroup generated by $S$. \\
By \cite[Theorem 3.3]{11} and \cite[Theorem 4.2]{11}, the group $\mathbf{Z}_{p} \oplus \mathbf{Z}_{p^2}$ has $\frac{p^2-1}{p-1}=p+1$ subgroups of order $p^2$, namely
\begin{align*}
J_{k}&= \left \langle \begin{pmatrix} k \\ 1 \end{pmatrix}\right \rangle, k \in \mathbf{Z}_{p}; \\
J_{p}&= \left \langle \begin{pmatrix} 1 \\ 0 \end{pmatrix} , \begin{pmatrix} 0 \\ p \end{pmatrix} \right \rangle.
\end{align*}
We first find the number of automorphisms of $\mathbf{Z}_{p} \oplus \mathbf{Z}_{p^2}$ which fix at least $p^2$ elements. The strategy we will follow is to find the number of automorphisms that fix, at least, the above $p+1$ subgroups of order $p^2$ and then substract the identity map. \\

The next theorem will play a crucial role in the proof of Conjecture \ref{conj1}. This theorem is proved in \cite[Theorem 3.2]{1}; however, for the reader's convenience, we include the statement.

\begin{theorem} \cite[Theorem 3.2]{1} \label{theo1} Let $H$ and $K$ be finite abelian groups, where $H$ and $K$ have no common direct factor. Then
\[
\operatorname{Aut}(H \times K)=
\left\{
\begin{pmatrix}
\alpha & \beta \\
\gamma & \delta 
\end{pmatrix}: \alpha \in \operatorname{Aut}H, \beta \in \operatorname{Hom}(K,Z(H)), \gamma \in \operatorname{Hom}(H,Z(K)), \delta \in \operatorname{Aut}K
\right\}
\]
\end{theorem}
where $Z(G)$ denotes the center of the group $G$. \\

We will also make use of the following lemmas.
\begin{lemma} \label{lemma1}There exists an isomorphism of abelian groups $\operatorname{Hom}_{\mathbf{Z}}(\mathbf{Z}_{n},\mathbf{Z}_{m}) \cong \mathbf{Z}_{(m,n)}$.
\end{lemma}
\begin{proof} This is a standard fact; see, for example, \cite[p.118]{10}.
\end{proof}

\begin{lemma} \label{lemma2} The finite indecomposable abelian groups are exactly the $p$-cyclic groups.
\end{lemma}

\begin{proof} This follows at once from the fundamental theorem of finitely generated abelian groups.
\end{proof}

\begin{lemma} \label{lemma3} The group $\operatorname{Aut}\mathbf{Z}_{n}$ is isomorphic to the group of units of $\mathbf{Z}_{n}$.
\end{lemma}

\begin{proof} It suffices to consider the evaluation map $T: \operatorname{Aut}\mathbf{Z}_{n} \rightarrow \mathbf{Z}_{n}^{\ast}$ given by $T(\alpha)=\alpha(1)$.
\end{proof}

Using Lemmas \ref{lemma1}, \ref{lemma2}, \ref{lemma3} and Theorem \ref{theo1}, one obtains that the group $\operatorname{Aut}(\mathbf{Z}_{p} \oplus \mathbf{Z}_{p^2})$ may be realized as the following set of $2 \times 2$ matrices
\begin{center}
$\left\{ \begin{pmatrix} a & b \\ pc & d \end{pmatrix}: a \in \mathbf{Z}_{p}^{\ast}, d \in \mathbf{Z}_{p^2}^{\ast}, b,c \in \mathbf{Z}_{p}\right\}$
\end{center}

\begin{remark} \label{rem2}
It follows from the above characterization that $|\operatorname{Aut}(\mathbf{Z}_{p} \oplus \mathbf{Z}_{p^2})|=\varphi(p)\varphi(p^2)p^2=p^{3}(p-1)^{2}$. 
\end{remark}

We now find the number of automorphisms of $\mathbf{Z}_{p} \oplus \mathbf{Z}_{p^2}$ fixing $p^2$ elements.
\begin{proposition} \label{prop1} We have $\theta(\mathbf{Z}_{p} \oplus \mathbf{Z}_{p^2},p^2)=p^3-p-1$.
\end{proposition}

\begin{proof} a) We first deal with $J_{0}=\left \langle \begin{pmatrix} 0 \\ 1 \end{pmatrix} \right \rangle$. Let $\phi=\begin{pmatrix} a & b \\ pc & d \end{pmatrix}$ be an automorphism of $\mathbf{Z}_{p} \oplus \mathbf{Z}_{p^2}$ such that $\phi \begin{pmatrix} 0 \\ 1 \end{pmatrix}= \begin{pmatrix} 0 \\ 1 \end{pmatrix}$. In this case we get that $b=0$ and $d=1$; so it suffices to choose $a$ and $c$. There are $\varphi(p)=p-1$ choices for $a$; and $p$ choices for $c$. Hence we get $p(p-1)$ automorphisms. 

b) Now we treat the case $J_{k}=\left\langle \begin{pmatrix} k \\ 1 \end{pmatrix}\right \rangle$, where $k \in \mathbf{Z}_{p} \setminus \{0\}$. Let $\phi=\begin{pmatrix} a & b \\ pc & d \end{pmatrix} \in \operatorname{Aut}(\mathbf{Z}_{p} \oplus \mathbf{Z}_{p^2})$ be such that $\phi \begin{pmatrix} k \\ 1 \end{pmatrix}= \begin{pmatrix} k \\ 1 \end{pmatrix}$. Then we get the following pair of congruences
\begin{equation}
\begin{cases}
&ak+b \equiv k \pmod{p} \\
&pck+d \equiv 1 \pmod{p^2}
\end{cases}
\end{equation}
Here we get $p\varphi(p)$ automorphisms and since there are $p-1$ possible choices for $k$, then we obtain $p\varphi(p)(p-1)=p(p-1)^2$ automorphisms.

c) Finally, we deal with $J_{p}= \left \langle \begin{pmatrix} 1 \\ 0 \end{pmatrix} , \begin{pmatrix} 0 \\ p \end{pmatrix} \right \rangle$. We first find the form of those automorphisms that fix at least $\left \langle \begin{pmatrix} 1 \\ 0 \end{pmatrix} \right \rangle$; then those that fix at least $\left \langle \begin{pmatrix} 0 \\ p \end{pmatrix} \right \rangle$ and take the intersection of these sets of matrices. Suppose first that $\phi$ is an automorphism such that $\phi \begin{pmatrix} 1 \\0 \end{pmatrix} = \begin{pmatrix} 1 \\ 0 \end{pmatrix}$. This gives the following pair of congruences
\begin{equation}
\begin{cases}
&a \equiv 1 \pmod{p} \\
&pc \equiv 0 \pmod{p^2}
\end{cases}
\end{equation}
Hence $a=1$ and $c=0$. Thus $\phi$ has the form $\begin{pmatrix} 1 & b \\ 0 & d \end{pmatrix}$, where $b \in \mathbf{Z}_{p}$ and $d \in \mathbf{Z}_{p^2}^{\ast}$.
Now suppose that $\phi$ is such that $\phi \begin{pmatrix} 0 \\p \end{pmatrix} = \begin{pmatrix} 0 \\ p \end{pmatrix}$. In this case, we get the following pair of congruences
\begin{equation}
\begin{cases}
&bp \equiv 0 \pmod{p} \\
&dp \equiv p \pmod{p^2}
\end{cases}
\end{equation}
Therefore $b$ is free and $d \equiv 1 \pmod{p}$. In this case $\phi$ has the following form $\begin{pmatrix} a & b \\ pc & jp+1 \end{pmatrix}$, where $a \in \mathbf{Z}_{p}^{\ast}$, $b,c,j \in \mathbf{Z}_{p}$. Now define
\begin{align*}
A&=\left\{ \begin{pmatrix} 1 & b \\ 0 & d \end{pmatrix}: b \in \mathbf{Z}_{p}, d \in \mathbf{Z}_{p^2}^{\ast} \right\} \\
B&=\left\{ \begin{pmatrix} a & b \\ pc & jp+1 \end{pmatrix}: a \in \mathbf{Z}_{p}^{\ast}, b,c,j \in \mathbf{Z}_{p} \right\}
\end{align*}
Then $A \cap B= \left\{ \begin{pmatrix} 1 & b \\ 0 & jp+1 \end{pmatrix}: b,j \in \mathbf{Z}_{p}\right\}$. Consequently, we obtain $p^2$ automorphisms.
Using a), b) and c) we obtain
\begin{align*}
\theta(\mathbf{Z}_{p} \oplus \mathbf{Z}_{p^2},p^2)&=p(p-1)-1+p(p-1)^2-(p-1)+p^2-1 \\
&=p^{3}-p-1.
\end{align*}
\end{proof}

\begin{lemma} \label{lemma4} Let $\alpha$ be a positive integer and let $p$ be a prime. Then $|\{a \in [1,p^{\alpha}-1]:(a,p^{\alpha})=(a-1,p^{\alpha})=1\}|=p^{\alpha}-2p^{\alpha-1}$.
\end{lemma}

\begin{proof} For simplicity, let us define $W=\{a \in [1,p^{\alpha}-1]: (a,p^{\alpha})=(a-1,p^{\alpha})=1\}$, $X=\{a \in [1,p^{\alpha}-1]: (a,p^{\alpha})=1\}$ and $Y=\{a \in [1,p^{\alpha}-1]: a \equiv 1 \pmod p\}$. Then $|W|=|X \setminus Y|=\varphi(p^{\alpha})-p^{\alpha-1}=p^{\alpha}-2p^{\alpha-1}$.
\end{proof}

\begin{lemma} \label{lemma5} Let $G$ be a finite abelian group and let $\phi \in \operatorname{Aut}G$. Then $\phi$ is a fixed-point-free automorphism if and only if $\phi-id_{G}$ is an automorphism.
\end{lemma}

\begin{proof} Immediate.
\end{proof}

\begin{proposition} \label{prop2} The number of fixed-point-free automorphisms of $\mathbf{Z}_{p} \oplus \mathbf{Z}_{p^2}$ is equal to $p^3(p-2)^2$; that is, $\theta(\mathbf{Z}_{p} \oplus \mathbf{Z}_{p^2},1)=p^3(p-2)^2$.
\end{proposition}

\begin{proof} Let $\phi = \begin{pmatrix} a & b \\ pc & d \end{pmatrix}$ be an automorphism of $\mathbf{Z}_{p} \oplus \mathbf{Z}_{p^2}$. By Lemma \ref{lemma5}, $\phi$ is a fixed-point-free automorphism if and only if $\phi-id_{G}$ is an automorphism of $G$. Note that
\begin{center}
$\phi-id_{G}=\begin{pmatrix} a- 1 & b \\ pc & d-1 \end{pmatrix}$
\end{center}
and this will be an element of $\operatorname{Aut}(\mathbf{Z}_{p} \oplus \mathbf{Z}_{p^2})$ if and only if $(a-1,p)=1$ and $(d-1,p^{2})=1$. Therefore, we need to count the number of elements $a$ in $\mathbf{Z}_{p}^{\ast}$ such that $(a-1,p)=1$, and the number of $d \in \mathbf{Z}_{p^2}^{\ast}$ such that $(d-1,p^2)=1$. Applying Lemma \ref{lemma4} we obtain $p-2$ choices for $a$ and $p^2-2p$ choices for $d$. Since there are $p$ choices for $b$ and $p$ choices for $c$, we obtain $p^2(p^2-2p)(p-2)=p^3(p-2)^2$ fixed-point-free-automorphisms of $\mathbf{Z}_{p} \oplus \mathbf{Z}_{p^2}$, as claimed.
\end{proof}

\begin{proposition} \label{prop3} We have that $\theta(\mathbf{Z}_{p} \oplus \mathbf{Z}_{p^2},p)=p(2p^3-4p^2+1)$.
\end{proposition}

\begin{proof} Using Remark \ref{rem2} and Propositions \ref{prop1}, \ref{prop2}, we have
\begin{align*}
\theta(G=\mathbf{Z}_{p} \oplus \mathbf{Z}_{p^2},p)&=|\operatorname{Aut}G|-\theta(G,1)-\theta(G,p^2)-\theta(G,p^3) \\
&=p^3(p-1)^2-p^3(p-2)^2-(p^3-p-1)-1 \\
&=2p^4-4p^3+p \\
&=p(2p^3-4p^2+1)
\end{align*}
\end{proof}
Since the identity map is the unique automorphism fixing the entire group, then $\theta(\mathbf{Z}_{p} \oplus \mathbf{Z}_{p^2},p^3)=1$.
Combining Propositions \ref{prop1}, \ref{prop2} and \ref{prop3} we obtain the following
\begin{theorem} \label{theo2} For the group $G=\mathbf{Z}_{p} \oplus \mathbf{Z}_{p^2}$, where $p$ is any prime: \\
\[
\theta(G,d)=
\begin{cases}
p^3(p-2)^{2} & \text{for} \ d=1; \\
p(2p^3-4p^2+1) & \text{for} \ d=p; \\
p^3-p-1 & \text{for} \ d=p^2; \\
1 & \text{for} \ d=p^3.
\end{cases}
\]
\end{theorem}
This proves Conjecture \ref{conj1}.

\begin{remark} \label{rem3} In the case $p=2$, we can also argue that $\theta(\mathbf{Z}_{2} \oplus \mathbf{Z}_{4},1)=0$ using the following argument based on the characterization of $\operatorname{Aut}(\mathbf{Z}_{2} \oplus \mathbf{Z}_{4})$. By \cite[Lemma 11.1]{4} we have that $\operatorname{Aut}(\mathbf{Z}_{2} \oplus \mathbf{Z}_{4}) \cong D_{8}$, the dihedral group of the square. Now, realize $D_{8}$ as the unitriangular matrix group of degree three over $\mathbf{Z}_{2}$ (a.k.a the Heisenberg group modulo $2$). It is easy to check that every matrix in this group has characteristic polynomial equal to $(t-1)^3$ and thus $1$ is always an eigenvalue. Therefore, no fixed-point-free automorphism of $\mathbf{Z}_{2} \oplus \mathbf{Z}_{4}$ exists.
\end{remark}

\section{Number of fixed-point-free automorphisms of $\mathbf{Z}_{p^{a}} \oplus \mathbf{Z}_{p^{b}}$} \label{section4}

Let $a$ and $b$ be positive integers with $a<b$. By \cite[Theorem 3.2]{1}, $\operatorname{Aut}(\mathbf{Z}_{p^{a}} \oplus \mathbf{Z}_{p^{b}})$ may be realized as the following set of $2 \times 2$ matrices: \\

\begin{center}
$\left \{ \begin{pmatrix} \alpha & \beta \\ cp^{b-a} & \delta \end{pmatrix}: \alpha \in \mathbf{Z}_{p^{a}}^{\ast}, \beta,c \in \mathbf{Z}_{p^{a}}, \delta \in \mathbf{Z}_{p^{b}}^{\ast} \right\}$
\end{center}

\begin{proposition} \label{prop4} Let $a$ and $b$ be positive integers with $a<b$. The number of fixed-point-free automorphisms of $\mathbf{Z}_{p^{a}} \oplus \mathbf{Z}_{p^{b}}$ is equal to:
\begin{center}
$\theta(\mathbf{Z}_{p^{a}} \oplus \mathbf{Z}_{p^{b}},1)=p^{3a+b-2}(p-2)^2.$
\end{center}
\end{proposition}

\begin{proof} Let $\phi$ be an automorphism of $G=\mathbf{Z}_{p^{a}} \oplus \mathbf{Z}_{p^{b}}$. Then, by Lemma \ref{lemma5}, $\phi$ is a fixed-point-free automorphism if and only if $\phi-id_{G}$ is an automorphism. Therefore, it suffices to compute the cardinalities of the following sets
\begin{align*}
&A=\left\{\alpha \in \mathbf{Z}_{p^{a}}^{\ast}: (\alpha-1,p^{a})=1\right\} \\
&B= \left\{\delta \in \mathbf{Z}_{p^{b}}^{\ast}: (\delta-1,p^{b})=1\right\}
\end{align*}
Applying Lemma \ref{lemma4} yields
\begin{align*}
|A|&=p^{a}-2p^{a-1}=p^{a-1}(p-2) \\
|B|&=p^{b}-2p^{b-1}=p^{b-1}(p-2)
\end{align*}
Since there are $p^{a}$ possible choices for each $\beta$ and $c$, then the total number of fixed-point-free automorphisms is equal to
\begin{center}
$p^{2a}p^{a-1}(p-2)p^{b-1}(p-2)=p^{3a+b-2}(p-2)^2$
\end{center} 
and the proof is now complete.
\end{proof}

\begin{remark} \label{rem4} The above result implies that all the groups $\mathbf{Z}_{2^{a}} \oplus \mathbf{Z}_{2^{b}}$, where $1 \leq a<b$, do not admit fixed-point-free automorphisms. This generalizes the result mentioned in Remark \ref{rem3}.
\end{remark}

\section{$\theta$ values for the dihedral group $D_{2p}$} \label{section5}

The description of $D_{2n}$ is obtained by using its generators: a rotation $a$ of order $n$ and a reflection $b$ of order
$2$. Using rotations and reflections, we can write a presentation of $D_{2n}$
\begin{align*}
D_{2n} &= \langle a,b: a^{n}=b^{2}=(ba)^{2}=1 \rangle \\
&=\left \{1,a,a^{2},\ldots,a^{n-1},b,ab,a^{2}b,\ldots,a^{n-1}b\right\}.
\end{align*}
Recall that $\operatorname{Aut}D_{2n} \cong \operatorname{Hol}(\mathbf{Z}_{n})$. More precisely,
$$\operatorname{Aut}D_{2n}=\{f_{\alpha, \beta}:\alpha \in \mathbf{Z}_{n}^{\ast}, \beta \in \mathbf{Z}_{n} \},$$
where $f_{\alpha, \beta}: D_{2n}\rightarrow D_{2n}$ is defined by
$$f_{\alpha, \beta}(a^{i}) = a^{\alpha i}$$ and
$$f_{\alpha, \beta} (a^{i}b) = a^{\alpha i + \beta}b,$$
for all $i=0,\ldots,n-1.$

\begin{remark} \label{rem5} The above implies that $|\operatorname{Aut}D_{2n}|=n\varphi(n)$.
\end{remark}
\begin{example}
For $n=4$, the dihedral group $D_{8}$ is the group of symmetries of the square.
$$D_{8}=\langle a,b : a^{4}=b^{2}=1 , bab=a^{-1}\rangle=\{1, a, a^{2},a^{3}, b, ab, a^{2}b,
a^{3}b\}$$ and its automorphism group is 
\begin{eqnarray*}
\operatorname{Aut}D_{8} &=& \{f_{\alpha, \beta}:  \alpha \in \mathbf{Z}_{4}^{\ast},~~\beta \in \mathbf{Z}_{4}  \}\\
 &=& \{ f_{1,0}, f_{1,1}, f_{1, 2}, f_{1, 3}, f_{3,0} , f_{3, 1}, f_{3, 2}, f_{3, 3}\}.
\end{eqnarray*}
In the following table, the images of elements of $D_{8}$ under all
automorphisms are given.
\begin{center}
\begin{tabular}{|c|c|c|c|c|c|c|c|c|}
\hline
g & $f_{1,0}(g)$ & $f_{1,1}(g)$ & $f_{1,2}(g)$ & $f_{1,3}(g)$ & $f_{3,0}(g)$ & $f_{3,1}(g)$ & $f_{3,2}(g)$ & $f_{3,3}(g)$\\
\hline
$1$      & $1$      & $1$      & $1$      & $1$      & $1$      & $1$      & $1$      & $1$ \\
\hline
$a$      & $a$      & $a$      & $a$      & $a$      & $a^{3}$  & $a^3$      & $a^{3}$  & $a^{3}$ \\
\hline
$a^{2}$  & $a^{2}$  & $a^{2}$  & $a^{2}$  & $a^{2}$  & $a^{2}$  & $a^{2}$  & $a^{2}$  & $a^{2}$ \\
\hline
$a^{3}$  & $a^{3}$  & $a^{3}$  & $a^{3}$  & $a^{3}$  & $a$      & $a$      & $a$      & $a$ \\
\hline
$b$      & $b$      & $ab$     & $a^{2}b$ & $a^{3}b$ & $b$      & $ab$     & $a^{2}b$ & $a^{3}b$ \\
\hline
$ab$     & $ab$     & $a^{2}b$ & $a^{3}b$ & $b$      & $a^{3}b$ & $b$      & $ab$     & $a^{2}b$ \\
\hline
$a^{2}b$ & $a^{2}b$ & $a^{3}b$ & $b$      & $ab$     & $a^{2}b$ & $a^{3}b$ & $b$      & $ab$ \\
\hline
$a^{3}b$ & $a^{3}b$ & $b$      & $ab$     & $a^{2}b$ & $ab$     & $a^{2}b$ & $a^{3}b$ & $b$  \\
\hline
\end{tabular}
\end{center}
\end{example}

We now prove the main result of this section.
\begin{theorem} \label{theo3}
For $D_{2p}$, where $p$ is an odd prime, we have
\[
\theta(D_{2p}, d)=
\begin{cases}
0 &\text{if d=1};\\
p(p-2) &\text{if d=2};\\
p-1 &\text{if d=p};\\
1 &\text{if d=$2p$}.
\end{cases}
\]
\begin{proof} a) We first deal with the case $d=2p$. Since the identity map $f_{1, 0}$ is the unique automorphism that fixes the whole group $D_{2p}$, it follows that $\theta(D_{2p}, 2p)=1$.\\
b) Now let us treat the case $d=p$. Since the unique subgroup of $D_{2p}$ of order $p$ is given by $H=\langle a \rangle$, then $S_{p}^{D_{2p}}=\Omega_{D_{2p}}^{-1}(H)$. Let us show that
\begin{center}
$S_{p}^{D_{2p}}=\Omega_{D_{2p}}^{-1}(H)=\{f_{1,i}: i=1,\ldots,p-1\}.$
\end{center}
For this, let $1\leq i\leq p-1$ be any integer. Then $$f_{1, i}(a^{i})=a^{i}$$
and
$$f_{1, i}(a^{j}b)=a^{j+i}b\neq a^{i}b.$$It follows that $f_{1, i} \in S_{p}^{D_{2p}}$.\\
Now suppose that $f_{\alpha, \beta} \in S_{p}^{D_{2p}}$, then $f_{\alpha, \beta }$
fixes $H$. Hence $$f_{\alpha,
\beta}(a^{i})=a^{i}$$ but $$f_{\alpha, \beta}(a^{i}b)\neq a^{i}b.$$ Then it easily follows that $\alpha=1$ and $\beta \neq 0$. Hence $f_{\alpha, \beta}$ is of the form $f_{1, i}$, where $i=1,\ldots, p-1$. We conclude that $\theta(D_{2p},p)=|S_{p}^{D_{2p}}|=p-1$.\\

c) We now compute $\theta(D_{2p},2)$. Let $\alpha \in \mathbf{Z}_{p}^{\ast} \setminus \{1\}$ and $\beta \in \mathbf{Z}_{p}$, then we claim that there exists a unique $i$ such that $$f_{\alpha,
\beta}(a^{i})=a^{\alpha i }\neq a^{i}$$ and $$f_{\alpha,
\beta}(a^{i}b)=a^{i}b.$$ 

Indeed, take $i\equiv (1-\alpha)^{-1}\beta \pmod{p}$. Therefore, each $f_{\alpha, \beta}$ fixes $a^{i}b$ for such a unique $i$. That is, $f_{\alpha, \beta}$ fixes exactly two elements: $1$ and
$a^{i}b$. There are $p(p-2)$ such $f_{\alpha, \beta}$, where $\alpha \in \mathbf{Z}_{p}^{\ast} \setminus \{1\}$ and $\beta \in \mathbf{Z}_{p}$.\\
Therefore $$\theta(D_{2p}, 2)=p(p-2).$$

d) Finally, we compute the number of fixed-point-free automorphisms of $D_{2p}$. Using Remark \ref{rem5}, together with a), b) and c), we get
\begin{eqnarray*}
\theta(D_{2p}, 1) &=& p\varphi(p)-\sum_{d|2p, d\neq 1}\theta(D_{2p}, d)\\
&=& p(p-1)-(1+p-1+p(p-2))\\
&=& 0
\end{eqnarray*}
as claimed.
\end{proof}
\end{theorem}
\begin{example}
For $p=5$, we have $$D_{10}=\{1,a,a^{2}, a^{3}, a^{4}, b, ab, a^{2}b, a^{3}b,
a^{4}b \}$$ and
$$\operatorname{Aut}D_{10}=\{f_{\alpha,\beta}: \alpha \in \mathbf{Z}_{5}^{\ast}, \beta\in \mathbf{Z}_{5}\}.$$
Then,
$S_{1}^{D_{10}}=\{\}$; $S_{2}^{D_{10}}=\{f_{\alpha, \beta}: \alpha\in \{2,3,4\}, \beta \in \{0,1,2,3,4\}\}$; $S_{5}^{D_{10}}=\{f_{1,1},f_{1,2},f_{1,3},f_{1,4}\}$ and $
S_{10}^{D_{10}}=\{I\}$. Thus
\[
\theta(D_{10}, d)=
\begin{cases}
0 &\text{if d=1};\\
15 &\text{if d=2};\\
4 &\text{if d=5};\\
1 &\text{if d=$10$}.
\end{cases}
\]
\end{example}
\begin{remark} Theorem \ref{theo3} does not hold if $p$ is not prime. For example, consider $D_{8}$. One can see that $S_{1}^{D_{8}}=\{\}$; $S_{2}^{D_{8}}=\{f_{3,1},f_{3,3}\}$; $S_{4}^{D_{8}}=\{f_{1,1},f_{1,2},f_{1,3},f_{3,0},f_{3,2}\}$ and $S_{8}^{D_{8}}=\{I\}$. Hence $\theta(D_{8}, 2)=2\neq 4(4-2)$ and $\theta(D_{8},4)=5\neq 4-1$.
\end{remark}

Finally, we end the paper with the following \\

\textbf{Question}. Let $p$ be a prime number and $a_{1},\ldots,a_{n}$ be distinct positive integers. Is there an exact formula for $\theta \left(\displaystyle \bigoplus_{i=1}^{n} \mathbf{Z}_{p^{a_{i}}},d\right)$?

\section*{acknowledgments}
We thank the anonymous referees for their comments and careful reading of the manuscript.

\end{document}